\documentclass[12pt]{amsart}
\usepackage{fullpage}
\usepackage{url}

\usepackage[alphabetic,backrefs,lite]{amsrefs}
\usepackage{color}
\usepackage{mathrsfs}
\usepackage{mathrsfs}

\newtheorem{theorem}{Theorem}
\newtheorem{lemma}[theorem]{Lemma}
\newtheorem{proposition}[theorem]{Proposition}

\newtheorem{cor}[theorem]{Corollary}

\theoremstyle{definition}
\newtheorem{definition}[theorem]{Definition}
\newtheorem{example}[theorem]{Example}
\theoremstyle{remark}
\newtheorem{remark}[theorem]{Remark}

\newcommand{\bC}{\mathbb C}

\newcommand{\bF}{\mathbb F}

\newcommand{\bP}{\mathbb P}
\newcommand{\bQ}{\mathbb Q}
\newcommand{\bZ}{\mathbb Z}

\newcommand{\GL}{{\rm GL}}

\newcommand{\Gal}{{\operatorname{Gal}}}
\newcommand{\Pic}{{\operatorname{Pic}}}

\newcommand{\Fbar}{\overline{\mathbb{F}}}
\newcommand{\Xbar}{\overline{X}}
\newcommand{\Fr}{{\rm{Fr}}}
\newcommand{\Tr}{{\rm{Tr}}}
\newcommand{\et}{\textup{\'et}}

\newcommand{\hideqed}{\renewcommand{\qed}{}} 

\DeclareMathOperator{\Proj}{{Proj}}

\begin{document}

\title[Unirationality of del Pezzo surfaces of degree two]{On the Unirationality of \\ del Pezzo surfaces of degree two}
\author{Cec\'ilia Salgado}
\address{Instituto de Matem\'atica, Universidade Federal do Rio de Janeiro, Ilha do Fund\~ao, 21941-909, Rio de Janeiro, Brasil}
\email{salgado@im.ufrj.br}
\urladdr{http://www.im.ufrj.br/\~{}salgado}

\author{Damiano Testa}
\address{Mathematics Institute, University of Warwick, Conventry, CV4 7AL, United Kingdom}
\email{adomani@gmail.com}
\urladdr{http://www.warwick.ac.uk/\~{}maskal/zone}

\author{Anthony V\'arilly-Alvarado}
\address{Department of Mathematics, Rice University, Houston, TX 77005, USA}
\email{varilly@rice.edu}
\urladdr{http://www.math.rice.edu/\~{}av15}

\thanks{The first author was partially supported by the stipendium Hendrik Casimir Karl Ziegler from the KNAW. The third author was partially supported by National Science Foundation Grant 1103659.}

\begin{abstract}
Among geometrically rational surfaces, del Pezzo surfaces of degree two over a field $k$ containing at least one point are arguably the simplest that are not known to be unirational over $k$.  Looking for $k$-rational curves on these surfaces, we extend some earlier work of Manin on this subject. We then focus on the case where $k$ is a finite field, where we show that all except possibly three explicit  del Pezzo surfaces of degree two are unirational over $k$.
\end{abstract}

\maketitle

\section{Introduction}

A variety $X$ over a field $k$ is \emph{rational} if it is birational over $k$ to projective space $\bP^m$; it is \emph{unirational} if there is a dominant map $\bP^m \dasharrow X$ defined over $k$.  Clearly, rational varieties are unirational, and the converse is called the L\"uroth problem.  For varieties of dimension one, a curve is unirational if and only if it is rational, over any field.  For varieties of dimension two, the same holds over algebraically closed fields of characteristic zero.  Artin~\cite{A} showed that there exist unirational K3 surfaces over algebraically closed fields of characteristic two. Shortly thereafter, Shioda~\cite{Shioda} proved that the Fermat quartic surface over a field of characteristic $p \equiv 3 \bmod 4$ is unirational; being a K3 surface, this surface is not rational.  In the early 1970s, Artin and Mumford~\cite{AM}, Clemens and Griffiths~\cite{CG}, and Iskovskikh and Manin~\cite{IM} gave three different constructions of threefolds over the complex numbers that are unirational but not rational.

Let $X$ be a smooth projective variety.  We are interested in criteria for unirationality for the variety $X$.  Let us assume that $X$ has a rational point; otherwise $X$ is clearly not unirational.  With this assumption, it is not known if \emph{geometrically} rational surfaces (that is, surfaces $X$ such that $\overline{X} = X\times_k \overline{k}$ is rational over an algebraic closure $\overline{k}$ of $k$) are unirational over their field of definition. Unirationality is a birational property, and a theorem of Iskovskikh~\cite{I} guarantees that every smooth projective geometrically rational surface is birational over the ground field to either a del Pezzo surface or a conic bundle.  Work of Segre, Manin and Koll\'ar shows that del Pezzo surfaces of degree $d \geq 3$ over any field $k$ are unirational, provided $X(k) \neq \emptyset$.

In~\cite{M}*{Theorem~29.4}, Manin proves that many del Pezzo surfaces of degree two are unirational.  Given a rational point $p$ that avoids an explicit divisor of $X$, Manin produces a rational curve $C$ through $p$; repeating the construction on the points of $C$, he produces a unirational parametrization of $X$.  The goal of this paper is to extend Manin's result on these surfaces, as well as to clarify an oversight in the explicit divisor that must be avoided in Manin's original work (see Corollary~\ref{fima}).

We give conditions to detect rational curves on del Pezzo surfaces of degree two, and thus prove unirationality of these surfaces. For example, we show that if $X$ is a del Pezzo surface of degree two over a field $k$, and if $X$ contains eight points whose images under the morphism defined by the anticanonical linear system $\kappa\colon X \to \bP^2$ are distinct and avoid the branch locus, then there is a rational curve over $k$ passing through one of these eight points, which implies $X$ is unirational (see Lemma~\ref{otto}). These sufficient conditions, together with some analysis of the Galois representation $\Gal(\overline{k}/k) \to W(E_7)$ associated to $X$, and a few auxiliary geometric lemmas, allow us to prove our main result.

Recall that all del Pezzo surfaces of degree two are smooth quartic surfaces in the weighted projective space $\bP_k(1,1,1,2) := \Proj\left(k[x,y,z,w]\right)$.

\begin{theorem} \label{thm:main}
Let $X$ be a del Pezzo surface of degree two over a finite field $\mathbb{F}$.  The surface $X$ is unirational except possibly in the following cases 
\[
\begin{array}{lrcl}
X_1/\bF_3 \colon & -w^2 & = & (x^2 + y^2)^2 + y^3z - yz^3, \\[5pt]
X_2/\bF_3 \colon & -w^2 & = & x^4 + y^3z - yz^3 , \\[5pt]
X_3/\bF_9 \colon & \nu w^2 & = & x^4 + y^4 + z^4,\quad \textup{where } \nu \in \bF_9 {\textup{ is a non-square.}}
\end{array}
\]
\end{theorem}

\begin{remark}
The three exceptional surfaces in Theorem~\ref{thm:main} are minimal over their field of definition; in fact $\Pic(X) \cong \bZ$, generated by the class of an anticanonical divisor ${-K}_X$. Hence, the only places to look for curves defined over $k$ are the linear systems $|{-nK_X}|$ for $n \in \mathbb{N}$. Exhaustive computer searches show that the linear systems $|{-nK}_X|$ contain no geometrically integral curves of geometric genus zero for $n \leq 3$ for the surfaces $X_1$ and $X_2$ and for $n\leq 2$ for the surface $X_3$.  We also note that these surfaces have only a few points, and these points all lie on the ramification divisor of $\kappa$: the surface $X_1$ has one point, $X_2$ has four points, and $X_3$ has $28$ points.   Up to isomorphism, the surface $X_1$ is the unique del Pezzo surface of degree two over a finite field containing exactly one point~\cite{Shuijing}.
\end{remark}

We obtain the following amusing corollary, which for infinite fields already follows easily from Manin's work.

\begin{cor}\label{quadratica}
Let $X$ be a del Pezzo surface of degree two defined over a field $k$. Then there is a quadratic extension $k'/k$ such that $X \times_{k} k'$ is unirational. 
\end{cor}

The paper is organized as follows.  Section~\ref{geo} provides background on del Pezzo surfaces, dealing mostly with curves of small degree and their configurations.  Section~\ref{sec:uni} contains the main unirationality results over arbitrary fields.  Section~\ref{sec:unifin} specializes to the case of del Pezzo surfaces of degree two over finite fields.  In the Appendix we give normal forms for del Pezzo surfaces of degree two containing a generalized Eckardt point over fields of characteristic two.

\subsection*{Acknowledgements}
We thank Brendan Hassett, Marc Hindry, Samir Siksek, Ronald van Luijk, Olivier Wittenberg for useful conversations.  We would also like to thank the Mathematisch Instituut in Leiden, the Max Planck Institute in Bonn, the Mathematics Department of Rice University, and CIRM for their support and hospitality at various stages of this project.

\section{Geometry of del Pezzo surfaces of degree two} \label{geo}
Let $X$ be a del Pezzo surface defined over a field $k$.  Denote by $K_X$ a canonical divisor on $X$ and by $\kappa \colon X \to \mathbb{P} \bigl( {\rm H}^0(X,\mathcal{O}_X({-K_X}))^\vee \bigr) \simeq \mathbb{P}^2$ the morphism induced by the anticanonical divisor class on $X$.  The morphism $\kappa$ is finite and separable of degree two; let $R \subset X$ be the ramification divisor and let $B \subset \mathbb{P}^2$ be the branch divisor of $\kappa$.  Since the degree of $\kappa$ is two, there is an involution $\varphi \colon X \to X$ commuting with $\kappa$, called the {\em{Geiser involution}}.  The curve $B$ is a plane quartic, and it is smooth if the characteristic of the field $k$ is not two; otherwise, the curve $B$ is a double conic, and the conic itself may be singular, or even a double line.  From now on, we will simply mention the ramification divisor or the branch divisor omitting reference to the anticanonical morphism.

\subsection{Bitangent lines}
A {\em{bitangent line}} to the branch curve $B$ is a line $\ell$ in $\mathbb{P}^2$ such that $\kappa^{-1}(\ell)$ is reducible.  If the characteristic of the field $k$ is not two, then a line whose intersection multiplicity with $B$ is even everywhere is a bitangent line.  The following lemma analyzes the reducible elements of the linear system $|{-K_X}|$, showing that they correspond to exactly the bitangent lines.

\begin{lemma} \label{lem:comp}
Let $X$ be a del Pezzo surface of degree two.  The non-integral elements of the linear system $|{-K_X}|$ are of the form $C_1 \cup C_2$, where $C_1,C_2$ are exceptional curves satisfying $C_1 \cdot C_2 = 2$.  These non-integral elements are exactly the inverse images of the bitangent lines to $B$.
\end{lemma}

\begin{proof}
Let $C$ be a non-integral element of the linear system $|{-K_X}|$.  Since the divisor ${-K_X}$ is ample, it has intersection number at least one with each integral component of $C$, and since the equality $-K_X \cdot C = 2$ holds, it follows that $C=C_1+C_2$, where $C_1,C_2$ are integral curves with $-K_X \cdot C_1 = -K_X \cdot C_2 = 1$.  We deduce that for $i = 1, 2$, the divisor $D := 2C_i + K_X$ is orthogonal to $K_X$; by the Hodge index theorem it follows that $(D_i)^2 = 4(C_i)^2-2 \leq 0$. This shows that $(C_i)^2 \leq 0$. The adjunction formula shows that $(C_i)^2 = 0$ is not possible, so $(C_i)^2 < 0$.  We conclude that $C_1$ and $C_2$ are exceptional curves.  Moreover, we have $2 = (-K_X)^2 = (C_1+C_2)^2 = 2(C_1 \cdot C_2 - 1)$, so that $C_1 \cdot C_2 = 2$, as required.
\end{proof}

\subsection{Generalized Eckardt points}\label{ss:Eckardt}
We begin this subsection with a result that is certainly well-known.

\begin{lemma} \label{lem:quattro}
Let $X$ be a del Pezzo surface of degree two.  Through any point of $X$ there are at most four exceptional curves.
\end{lemma}

\begin{proof}
A proof can be found in~\cite{TVAV}*{proof of Lemma~4.1}.
\end{proof}

Taking a cue from the theory of cubic surfaces (which are del Pezzo surfaces of degree three), we make the following definition.

\begin{definition}
A \emph{generalized Eckardt point} is a point on a del Pezzo surface of degree two contained in four exceptional curves.  
\end{definition}

We give an upper bound for the number of generalized Eckardt points that can occur. Let $\mathscr{S}$ be a set of $28$ lines in $\mathbb{P}^2$. Each line in $\mathscr{S}$ contains at most $\frac{28-1}{3} = 9$ points through which there are at least four lines all contained in $\mathscr{S}$.  Thus there are at most $28 \cdot 9$ pairs $(\ell , p)$ consisting of a line $\ell$ in $\mathscr{S}$ together with a point $p$ in $\ell$ contained in four of the lines of $\mathscr{S}$; finally, there are at most $\frac{28 \cdot 9}{4} = 63$ points each contained in four lines of $\mathscr{S}$.  

Applying the above count to the set $\mathscr{S}$ of $28$ bitangent lines to the branch curve of a del Pezzo surface $X$ of degree two, we conclude that $X$ has at most $2 \cdot 63 = 126$ generalized Eckardt points. This upper bound is achieved by the surface with equation 
\[
X/\bF_9 \colon \quad
w^2 = x^4 + y^4 + z^4;
\]
the $126$ generalized Eckardt points on $X$ project via $\kappa$ to the $63$ points of $\bP(\bF_9)$ not contained in the branch curve.

\begin{example}
Assume that the characteristic of the ground field is not two.  We construct del Pezzo surfaces of degree two with a point $p$ lying on four exceptional curves.  Note that the point $p$ cannot be in the ramification divisor $R$, since there are at most two exceptional curves through any point of $R$.  Let $q_2,q_4$ be homogeneous polynomials of degrees two and four respectively and let $F(x,y,z)$ be the polynomial $F = x^4 + q_2(y,z) x^2 + q_4(y,z)$.  Let $B \subset \mathbb{P}^2$ be the plane quartic with equation $F=0$.  If $B$ is smooth, namely if the polynomial $q_2^2-4q_4$ has distinct roots, then the surface $S$ in the weighted projective space $\mathbb{P}(1,1,1,2)$ with equation 
\[
S \colon \quad w^2=F(x,y,z)
\]
is a del Pezzo surface of degree two.  The point $p = [1,0,0] \in \mathbb{P}^2$ is contained in the four bitangent lines to $B$ defined by the vanishing of the linear factors of the polynomial $q_2^2-4q_4$.  The eight exceptional curves lying above the four bitangent lines to the quartic $B$ through the point $p$ decompose into two sets of four, according to which of the two points $p_\pm = [1,0,0,\pm1]$ above $p$ they contain.  Thus the two points $p_\pm$ are both generalized Eckardt points.  

The surfaces we constructed above all have an involution given by $x \mapsto -x$.  In fact, it follows from~\cite{Dolgachev}*{Exercise~6.17} that every del Pezzo surface of degree two with a point contained in four exceptional curves has an involution and is of the form described in this example.
\end{example}

\subsection{Spines}

In this subsection, we show that if $p$ is a point of the ramification divisor $R$ of a del Pezzo surface $X$ of degree two, then there is a unique section of $|{-K}_X|$ through $p$ that is singular at $p$.  These sections can be a source of rational curves on $X$, and can thus help build unirational parametrizations of $X$.

\begin{lemma} \label{lem:Kliscio}
Let $X$ be a del Pezzo surface of degree two defined over a field $k$ and let $p$ be a rational point of $X$.  There is at most one element of the linear system $|{-K_X}|$ that is singular at $p$.  There is such a singular element if and only if $p$ is contained in the ramification divisor $R$ of the morphism $\kappa$.
\end{lemma}

\begin{proof}
Since the anticanonical linear system is base point free, the subsystem $L_p \subset |{-K}_X|$ consisting of divisors containing $p$ is a line.  We separate the argument into two cases, according to whether $p$ is contained or not in $R$.

Suppose that $p$ is not contained in $R$.  Then the morphism $\kappa$ is \'etale at $p$ and since the image of any element of $|{-K_X}|$ is a line and hence smooth, we deduce that any element of $L_p$ is non-singular at $p$.

Suppose that $p$ is contained in the ramification divisor $R$.  Let $\overline{x},\overline{y}$ be a local system of parameters on $\mathbb{P}^2$ near $\overline{p} := \kappa(p)$ and let $x,y$ respectively denote the pull-back to $X$ of $\overline{x},\overline{y}$.  The morphism $\kappa$ makes the local ring $\mathcal{O}_p$ of $X$ at $p$ into a module for the local ring $\mathcal{O}_{\overline{p}}$ of $\mathbb{P}^2$ at $\overline{p}$.  Since the morphism $\kappa$ is finite of degree two, the $\mathcal{O}_{\overline{p}}$-module $\mathcal{O}_p$ is generated by $1$ and any element $z$ not in the submodule generated by $1$.  Thus there are elements $c,d \in \mathcal{O}_{\overline{p}}$ such that $z^2+cz+d = 0$.  If the characteristic of $k$ is different from two, then, replacing $z$ by $z-c/2$ we reduce to the case in which $c$ vanishes identically.  If the characteristic of $k$ is equal to two, then replacing $z$ by $z-z(p)$ we reduce to the case in which $z(p)$ vanishes; in this case $c^2$ is an equation of the ramification divisor $R$ near $p$, and hence $c(p)$ vanishes; by our reductions, $d(p)$ vanishes as well.  In either case, it suffices to assume that $z$, $c$ and $d$ all vanish at $p$.  Since the terms $z^2$ and $cz$ in the equation of $X$ vanish to order at least two at $p$, and since $X$ is non-singular at $p$, we deduce at once that the vanishing set of $d$ contains $p$ and is non-singular at $p$.  Let $D \subset \mathbb{P}^2$ denote the (image under $\kappa$ of the) vanishing set of $d$, so that $D$ contains $\overline{p}$ as a non-singular point.  It is now immediate to check that a line $\ell$ in $\mathbb{P}^2$ through $\overline{p}$ determines an element of $L_p$, non-singular at $p$, if and only if $\ell$ is not tangent to $D$ at $\overline{p}$.  Thus there is exactly one element of $L_p$ singular at $p$ and the lemma follows.
\end{proof}

It follows from Lemma~\ref{lem:Kliscio} that for every point $p$ of the ramification divisor $R$ there is a well-defined one-dimensional subspace $\mathcal{S}_p$ of the tangent space to $\mathbb{P}^2$ at $\kappa(p)$ with the following property.  A line $\ell$ in $\mathbb{P}^2$ containing $\kappa(p)$ determines an element of $|{-K_X}|$ singular at $p$ if and only if the tangent space to $\ell$ at $\kappa(p)$ is $\mathcal{S}_p$.  In the case in which the characteristic of $k$ is different from two, it is easy to check that the space $\mathcal{S}_p$ is in fact the tangent space to the branch curve $B$ at the point~$\kappa(p)$.

\begin{definition}
Let $X$ be a del Pezzo surface of degree two and let $p$ be a point of $X$ contained in the ramification divisor.  We call the section of the linear system $|{-K_X}|$ through $p$ and singular at $p$ the {\em{spine}} of $X$ at $p$.
\end{definition}

We shall also need the following lemma, which is a companion to Lemma~\ref{lem:Kliscio}.

\begin{lemma} \label{cor:irr}
Let $X$ be a del Pezzo surface of degree two defined over a field $k$ and let $p$ be a rational point of $X$.  If $k$ contains at least four elements, then there exists an integral element of the linear system $|{-K_X}|$ containing $p$ as a non-singular point.
\end{lemma}

\begin{proof}
Let $L_p$ denote the linear subsystem of $|{-K_X}|$ consisting of the divisors containing $p$.  By Lemma~\ref{lem:comp}, each non-integral element of $L_p$ consists of a union of two exceptional curves, and at least one of these two curves contains $p$.  Since different elements of $L_p$ have no component in common, we conclude from Lemma~\ref{lem:quattro} that there are at most four reducible elements in $L_p$.  If the field has at least four elements, then the linear system $L_p$ has at least five elements, and at least one of them is therefore not reducible.  The corollary follows.
\end{proof}

\subsection{Blow ups}
We fix notation for the remainder of this subsection. Let $X$ be a del Pezzo surface of degree two and let $p$ be a point on $X$.  Denote by $b \colon X' \to X$ the blow-up of $X$ at $p$ and by $K_{X'}$ a canonical divisor on $X'$.

\begin{theorem} \label{thm:coho}
Let $n$ be a non-negative integer; the equalities 
\[
\dim |{-nK_{X'}}| = \frac{n^2+n}{2}
\quad \quad {\textrm{and}} \quad \quad 
h^1 (X', \mathcal{O}_{X'} (-nK_{X'})) = 0
\]
hold.  Moreover, the linear system $|{-nK_{X'}}|$ has a unique reduced base point if $n=1$ and is base point free otherwise.
\end{theorem}

\begin{proof}
If $X'$ is a del Pezzo surface, the result is well-known; for instance, see~\cite{Ko}*{Corollary~III.3.2.5}.

In any case, it suffices to prove the result after an extension of the base field.  Let $L_p$ denote the linear subsystem of $|{-K_X}|$ consisting of all divisors containing $p$.  Extending the base field if necessary, we shall assume that there is an irreducible element in $|{-K_X}|$ containing $p$ and non-singular at $p$ (cf.\ Lemma~\ref{cor:irr}).

Denote by $E$ the exceptional divisor of $b$, and let $C$ be an element of $L_p$, so that the effective divisor $C' := b^{*}C-E$ is an element of $|{-K_{X'}}|$.  Using the adjunction formula, we find that the dualizing sheaf on $C'$ is trivial and hence the arithmetic genus of $C'$ equals one.  If $C$ is integral and non-singular at $p$, then the corresponding divisor $C'$ is also integral (and non-singular at $C' \cap E$).  We deduce that the linear system $|{-K_{X'}}|$ contains integral divisors and since it has dimension at least one, we deduce that it has no base components.

We proceed by induction on $n$.  The case $n=0$ is clear since $X'$ is a rational surface and hence the group ${\rm H}^1(X' , \mathcal{O}_{X'})$ vanishes by Castelnuovo's Rationality Criterion.  Suppose that $n\geq 1$ and that the result is true for smaller values of $n$.  Since the canonical divisor is not linearly equivalent to an effective divisor, an application of Serre duality shows that for every effective divisor $F$ on $X'$ the group ${\rm H}^2 \bigl( X' , \mathcal{O}_{X'} (F) \bigr)$ vanishes.  Let $D \in |{-K_{X'}}|$ be an integral element.  The long exact cohomology sequence associated to the sequence 
\begin{equation} \label{defD}
0 \longrightarrow \mathcal{O}_{X'} (-(n-1) K_{X'}) \longrightarrow \mathcal{O}_{X'} (-n K_{X'}) \longrightarrow 
\mathcal{O}_D (-n K_{X'})  \longrightarrow 0
\end{equation}
induces a short exact sequence of global sections by the inductive hypotheses and it also induces an isomorphism ${\rm H}^1 \bigl( X' , \mathcal{O}_{X'} (-nK_{X'}) \bigr) \simeq {\rm H}^1 \bigl( D , \mathcal{O}_D (-nK_{X'}) \bigr)$.  The line bundle $\mathcal{L}_n := \mathcal{O}_D (-nK_{X'})$ on $D$ has degree $n>0$ and, since $D$ is an integral curve of arithmetic genus one, we conclude that the group ${\rm H}^1 (D , \mathcal{L}_n)$ vanishes and that the dimension of ${\rm H}^0 (D , \mathcal{L}_n)$ is $n$.  In particular, the formulas for the dimensions of the various cohomology groups follow by induction.

Finally, we analyze the base points.  If the linear system associated to $\mathcal{L}_n$ has base points, then it follows that $n=1$, and the base point is a single point.  Using the exact sequence on global sections of the exact sequence of sheaves above, we conclude that the linear system $|{-nK_{X'}}|$ has no base points if $n>1$.  For the case $n=1$, there is a unique reduced base point since $|{-K_{X'}}|$ has dimension one and the equality $(-K_{X'})^2=1$ holds.
\end{proof}

\begin{theorem} \label{thm:sing}
Let $C\subset X'$ be an integral curve such that $({-K}_{X'})\cdot C = 0$. Then either $C$ is the strict transform of an exceptional curve on $X$ passing through $p$, or the point $p$ lies on the ramification divisor and $C$ is a component of the strict transform of the spine of $X$ at $p$.  In particular, every integral curve $C \subset X'$ such that $(-K_{X'}) \cdot C = 0 $ is a $(-2)$-curve.
\end{theorem}

\begin{proof}
We work over an algebraic closure of the field of definition of $X$.  The linear system $|{-K_{X'}}|$ has irreducible general element by Lemma~\ref{cor:irr} and no base component by Theorem~\ref{thm:coho}.  Thus, if $C$ is an integral curve on $X'$ with $(-K_{X'}) \cdot C = 0$, then it follows that $C$ is an irreducible component of a reducible element of $|{-K_{X'}}|$.  Let $D$ be a reducible element of the linear system $|{-K_{X'}}|$.  If $D$ contains $E$, then $E \cdot (D-E) = E \cdot D + 1 = 2$, so that $b(D)$ is an element of $|{-K_X}|$ singular at $p$.  If $D$ does not contain $E$, then $b(D)$ is also reducible and it follows from Lemma~\ref{lem:comp} that the components of $D$ are strict transforms of exceptional curves on $X$.  It is now easy to check that the curves stated in the theorem do indeed have intersection number zero with ${-K_{X'}}$ and are $(-2)$-curves.
\end{proof}

\begin{remark}
If the point $p$ of Theorem~\ref{thm:sing} is such that $\kappa(p)$ is not contained in any bitangent line to $B$, then there is no reducible curve in $|{-K_X}|$ containing $p$ by Lemma~\ref{lem:comp}.  Moreover, there is a curve with zero intersection with $K_{X'}$ through $p$ if and only if $p$ is contained in $R$.  In this case, the curve is the spine of $X$ at $p$.  If the characteristic of $k$ is not two, then this curve is the strict transform of the tangent line to $B$ at $\kappa(p)$.

If the point $p$ of Theorem~\ref{thm:sing} is such that $\kappa(p)$ is contained in a bitangent line to $B$, then it can be contained in at most four bitangent lines by Lemma~\ref{lem:quattro}.
\end{remark}

\begin{cor}
\label{cor:correction}
Let $k$ be a field and let $X$ be a del Pezzo surface of degree two over $k$ with a rational point $p$.  The blow-up of $X$ at $p$ is a del Pezzo surface of degree one if and only if the point $p$ does not lie on any exceptional curve nor on the ramification divisor $R$. \qed
\end{cor}

\begin{remark}
If $X$ is a del Pezzo surface of degree $d \geq 3$, then the blow-up of $X$ at $p \in X(k)$ is a del Pezzo surface of degree $d-1$ if and only if the point $p$ does not lie on any exceptional curve.
\end{remark}

\section{Unirationality} \label{sec:uni}

In this section, we give conditions ensuring that a del Pezzo surface of degree two $X$ is unirational over a general field $k$; in the next section, we specialize to the case of finite fields.  If the surface $X$ is not minimal over $k$, then either~\cite{M}*{Theorem~29.4} or~\cite{kollar} is enough for this purpose.  On the other hand, the minimality of $X$ does not simplify our arguments and we therefore do not assume it.

Assuming the existence of a rational point $p$ on $X$, Manin constructs in~\cite{M} a rational curve on $X$ singular at $p$, provided the point $p$ is general.  Once this is done, it is easy to repeat this construction starting from each point of the rational curve and obtain a dominant rational map $\mathbb{P}^2 \to X$.  Manin assumes that the point $p$ does not lie on any exceptional curve, but for his argument to work he needs that the blow-up of $X$ at $p$ is a del Pezzo surface of degree one.  These two conditions on $p$ are not equivalent: if the blow-up of $X$ at $p$ is a del Pezzo surface of degree one, then the point $p$ does not lie on any exceptional curve, but the converse statement is not true; see Corollary~\ref{cor:correction}.

We analyze in detail Manin's argument: we prove that if $p$ is not in the union of the ramification curve of $X$ and of the points of $X$ lying on four exceptional curves, then the unirationality construction goes through (Corollary~\ref{fima}).  We begin with the following result (Theorem~\ref{ma2}) showing that the presence of a general point on the surface $X$ implies the existence of a rational curve on $X$.  We then prove that the presence of a rational curve is sufficient to prove unirationality, at least over fields of characteristic different from two (Theorem~\ref{unirat}).

\begin{theorem} \label{ma2}
Let $X$ be a del Pezzo surface of degree two defined over a field $k$ and let $p$ be a rational point on $X$.  If the point $p$ is not contained in four exceptional curves nor on the ramification divisor $R$, then there is a non-constant morphism $\mathbb{P}^1 \to X$ whose image is either an exceptional curve or contains the point $p$.
\end{theorem}

\begin{proof}
Let $b \colon X' \to X$ be the blow-up of $X$ at the point $p$ and let $E \subset X'$ denote the exceptional divisor of $b$.  Let $D$ be the divisor class $-2K_{X'}-E$ on $X'$.  We show that the linear system $|D|$ consists of a single point and that the unique effective divisor in $|D|$ contains a rational component.  We have the identities $D^2=K_{X'} \cdot D=-1$, so that the Riemann-Roch formula implies that at least one among the divisors $D$ and $K_{X'}-D$ is linearly equivalent to an effective divisor.  Since the divisor $-K_{X'}$ is nef by Theorem~\ref{thm:coho}, we deduce that $K_{X'}-D$ cannot be effective and hence $D$ is effective.  To avoid introducing more notation, we replace $D$ by an effective divisor in $|D|$.  Since the equality $-K_{X'} \cdot D = 1$ holds and $-K_{X'}$ is nef, it follows that $D=D_0+D_1$ where $D_1$ is irreducible and $-K_{X'} \cdot D_1=1$ and $-K_{X'} \cdot D_0=0$.

The divisor $D_0$ is a linear combination of the strict transforms of the exceptional curves of $X$ through the point $p$ by Theorem~\ref{thm:sing}.  Note that the exceptional curves through $p$ cannot have intersection number two with one another since otherwise the point $p$ would be contained in the ramification divisor $R$.  Thus, the strict transforms of these curves are pairwise disjoint and they form the components of $D_0$.  Hence the divisor $D_0$ is a non-negative combination of strict transforms of exceptional curves containing $p$; the divisor $D_0$ could equal zero, if there are no exceptional curves through $p$.  By hypothesis and Lemma~\ref{lem:quattro}, there are at most three such exceptional curves.  We may therefore write $D_0 = n_1 E_1 + n_2 E_2 + n_3 E_3$, where $E_1,E_2,E_3$ are strict transforms of exceptional curves on $X$ such that for $i \in \{1,2,3\}$ if $n_i$ is non-zero, then $E_i$ contains $p$.  We want to show that the divisor $D_0$ is reduced, or equivalently that $n_1,n_2,n_3$ are at most one.

Assume to the contrary that $D_0$ is not reduced; permuting if necessary the indices, we reduce to the case in which $n_1 \geq 2$.  From the identity $D = - 2 b^*(K_X) - 3E = D_1 + n_1E_1+n_2E_2+n_3E_3$, we deduce the equality 
\begin{equation} \label{impo}
2(- b^*(K_X) - E_1) = 3E + D_1 + (n_1-2)E_1+n_2E_2+n_3E_3
\end{equation}
of effective divisors on $X'$.  The divisor $- b^*(K_X) - E_1$ is the inverse image of the exceptional curve $\overline{E_1}$ on $X$ such that $b(E_1) + \overline{E_1} = -K_X$, so that $\overline{E_1}$ is the result of applying the Geiser involution to the exceptional curve corresponding to $E_1$.  The curve $\overline{E_1}$ does not contain the point $p$, since otherwise the point $p$ would be contained in the ramification divisor $R$.  It follows that the linear systems $|\overline{E_1}| = |{- b^*(K_X) - E_1}|$ and $|2\overline{E_1}| = |2(- b^*(K_X) - E_1)|$ contain a unique curve, and neither of these curves contains the curve $E$ as a component.  We conclude that the identity~\eqref{impo} is impossible and that $D_0$ is reduced.

As a consequence of what we just argued, we prove that the divisor $D_1$ cannot equal $E$.  Assume for a contradiction that the identity $-2K_{X'}-E=D_0+E$ holds; pushing this identity forward to $X$ we find $-2K_X = b_*(D_0)$ and $b_*(D_0)$ is a sum of at most three exceptional curves.  In particular, $b_*(D_0)$ has anticanonical degree at most three, whereas $-2K_X$ has anticanonical degree four.

To conclude the proof, it suffices to show that $D_1$ is a smooth rational curve.  By the adjunction formula and the identity $-K_{X'} \cdot D_1 = 1$, it suffices to show that $(D_1)^2=-1$ and this follows from the identity $(D_1)^2= (D-D_0)^2 = -1-2(n_1^2+n_2^2+n_3^2) - 2 (n_1+n_2+n_3) = -1$.
\end{proof}

\begin{theorem} \label{unirat}
Let $k$ be a field and $X$ a del Pezzo surface of degree two over $k$.  Suppose that $\rho \colon \mathbb{P}^1 \to X$ is a non-constant morphism; if the characteristic of the field $k$ is two and the image of $\rho$ is contained in the ramification divisor, then assume also that the field $k$ is perfect.  Then the surface $X$ is unirational.
\end{theorem}

\begin{proof}
The generic point of the image of the morphism $\rho$ is not contained in four exceptional curves, since the morphism $\rho$ is non-constant; if the image of $\rho$ is not contained in the ramification curve $R$, then we can apply Theorem~\ref{ma2} to the generic point of the image of $\rho$ and conclude.  If the characteristic of the field $k$ is different from two, the divisor $R$ is a smooth curve of genus three, and hence the image of $\rho$ cannot be contained in $R$.  Thus, we reduce to the case where the characteristic of the field $k$ is two and the image of $\rho$ is contained in the ramification divisor $R$; by assumption, the field $k$ is perfect.

First, we give an outline of the argument.  Let $S$ be the spine of $X$ at the generic point $\eta$ of the image of $\mathbb{P}^1$: the curve $S$ over $\eta$ is a geometrically integral curve of arithmetic genus zero.  The normalization $C$ of $S$ is a smooth rational curve defined over a purely inseparable extension of $\eta$ (possibly $\eta$ itself).  Composing, if necessary, the morphism $\rho$ with the Frobenius morphism, we reduce to the case in which the curve $C$ is defined over~$\eta$.  Thus, to conclude it suffices to show that $C$ has a rational point, or equivalently that there is a section to the morphism $C \to \eta$, since $C$ determines a one-parameter family of smooth rational curves over $\mathbb{P}^1$ covering the generic point of $X$.  If the degree of the geometric reduced branch curve is equal to $2$, the spine $S$ meets the ramification divisor at two distinct points: one determines the singular point of the spine, the other determines a purely inseparable section to the morphism $C \to \eta$.  Otherwise, the geometric branch curve is a double line, the normalization $C \to S$ is a unibranch morphism, and the reduced subscheme over the singular point of $S$ is a purely inseparable section of $C \to \eta$.  In both cases, we conclude by further composing $\rho$ with the Frobenius morphism, if necessary.

We now execute the above strategy in the case where the branch curve is a double line; this case presents all the technical difficulties that arise.  We work over the algebraic closure of the generic point of the image of the morphism $\rho$, since all we need to check is that the morphism $C \to S$ is geometrically unibranch.  Let $\mathbb{P}(1,1,1,2)$ be the weighted projective space with coordinates $x,y,z$ of weight one and $w$ of weight two.  After a change of coordinates, an equation for $X$ is 
\[
X \colon \quad w^2 + x^2 w + g(x,y,z) = 0
\]
where $g$ is a form of degree four, the ramification curve is the curve with equation $x=0$, and the geometric generic point $\eta$ of the image of the morphism $\rho$ is the point $p = [0,0,1,0]$.  The spine $S$ of $X$ at the point $p$ is neither contained in the ramification divisor, nor is it a union of exceptional curves.  In particular, an equation for the line $\kappa(S)$ is $y = \alpha x$, for some $\alpha$, and an equation of the spine $S$ in $\mathbb{P}(1,1,2)$ is 
\begin{equation} \label{stl}
S \colon \quad w^2 + x^2 w + g(x, \alpha x ,z) = 0.
\end{equation}
Dehomogenize with respect to $z$ and observe that the homogeneous component of degree two of equation~\eqref{stl} is a square.  Moreover, the curve defined by~\eqref{stl} is unibranch if and only if the spine $S$ is not a union of exceptional curves, and we conclude that this curve is indeed unibranch, as required.
\end{proof}

\begin{cor} \label{fima}
Let $X$ be a del Pezzo surface of degree two defined over a field $k$ and let $p$ be a rational point on $X$.  If the point $p$ is not contained in four exceptional curves nor on the ramification divisor $R$, then the surface $X$ is unirational.
\end{cor}

\begin{proof}
This follows from Theorems~\ref{ma2} and~\ref{unirat}.
\end{proof}

We conclude this section with two results that imply unirationality of certain del Pezzo surfaces of degree two by applying Corollary~\ref{fima} or Theorem~\ref{unirat}.  These results are will play a key role in the proof of Theorem~\ref{thm:main}.

\begin{lemma}[Eight points Lemma] \label{otto}
Let $X$ be a del Pezzo surface of degree two over a field $k$ and suppose that $X(k)$ contains 8 points $p_1,\ldots,p_8$ with the property that $\kappa(p_1),\ldots,\kappa(p_8)$ are distinct points not lying on the branch curve.  Then one of the points $p_1,\ldots,p_8$ is either contained in an exceptional curve defined over $k$, or it is not a generalized Eckardt point.  In particular, the surface $X$ is unirational.
\end{lemma}

\begin{proof}
If there is a point $p \in \{p_1,\ldots,p_8\}$ and a bitangent line $\ell$ in $\mathbb{P}^2$ defined over $k$ and containing the point $\kappa(p)$, then the irreducible component of $\kappa^{-1}(\ell)$ containing $p$ is an exceptional curve defined over $k$, as required.  Hence, we reduce to the case in which no bitangent line through $\kappa(p_1),\ldots,\kappa(p_8)$ is defined over the ground field.  It follows that no bitangent line contains more than one of the points $\kappa(p_1),\ldots,\kappa(p_8)$, since otherwise it would contain two rational points, and it would be defined over $k$.

We proceed by contradiction, and assume that the points $p_1,\ldots,p_8$ are all generalized Eckardt points and that there is no bitangent line defined over $k$ and containing one of the points $p_1,\ldots,p_8$.  By the above argument we deduce that for each point $p \in \{p_1,\ldots,p_8\}$ the four bitangent lines containing $\kappa(p)$ do not meet any of the points in $\{\kappa(q) : q \in \{p_1,\ldots,p_8\}, q \neq p \}$, and thus that there must be at least $4 \cdot 8 = 32$ bitangent lines.  This contradiction concludes the argument.
\end{proof}

\begin{lemma} \label{altra}
Let $a_1,\ldots,a_6 \in k$ be such that the equation 
\begin{equation} \label{e2inv}
w^2 = a_1^2 x^4 + a_2^2 y^4 + a_3^2 z^4 + (a_4 x^2 y^2 + a_5 x^2 z^2 + a_6 y^2 z^2)
\end{equation}
defines a del Pezzo surface $X$ of degree two in $\bP(1,1,1,2)$ and let $C$ be the smooth plane conic with equation $C \colon (a_4 - 2a_1 a_2) z^2 + (a_5 - 2 a_1 a_3) y^2 + (a_6 - 2 a_2 a_3) x^2 = 0$.  There is a morphism $f \colon C \to X$ that is birational to its image.  In particular, if the conic $C$ is isomorphic to $\mathbb{P}^1$ over $k$, then the surface $X$ is unirational.
\end{lemma}

\begin{proof}
Intersecting the surface $X$ and the surface with equation $w = a_1 x^2 + a_2 y^2 + a_3 z^2$ and eliminating the variable $w$ we obtain the plane quartic $Q$ with equation 
\[
Q \colon (a_4 - 2a_1 a_2) x^2y^2 + (a_5 - 2 a_1 a_3) x^2z^2 + (a_6 - 2 a_2 a_3) y^2z^2 = 0 .
\]
Observe that the coefficient $(a_4 - 2a_1 a_2)$ of the monomial $x^2y^2$ in the equation of $Q$ is nonzero, since otherwise the points $[a,b,0]$ satisfying $a_1a^2+a_2b^2=0$ are singular for the surface $X$; a similar argument shows that the coefficients of $x^2z^2,y^2z^2$ in the equation of $Q$ are also nonzero.  The quartic $Q$ has simple nodes at the three coordinate points, and is non-singular otherwise.  Applying the Cremona transformation $[x,y,z] \mapsto [\frac{1}{x} , \frac{1}{y} , \frac{1}{z}]$, changes the quartic $Q$ into the conic $C$ with equation $(a_4 - 2a_1 a_2) z^2 + (a_5 - 2 a_1 a_3) y^2 + (a_6 - 2 a_2 a_3) x^2 = 0$, as required.  The final statement follows from Theorem~\ref{unirat}.
\end{proof}

\begin{remark}
Changing the sign of any of the three elements $a_1,a_2,a_3 \in k$ determines four conics with a non-constant morphism to $X$: if any of these conics has a point, then the surface $X$ is unirational.

Over an algebraically closed field $k$ of characteristic different from two, a del Pezzo surface of degree two has an equation as in~\eqref{e2inv} if and only if its ramification curve admits two distinct commuting involutions.  Moreover, the points $[1,0,0,a_1]$, $[0,1,0,a_2]$, and $[0,0,1,a_3]$ are all generalized Eckardt points and the image of the conic $C$ in the surface $X$ has nodes at these points.
\end{remark}

\section{Unirationality over finite fields} \label{sec:unifin}

Let $\bF$ be a finite field with $q$ elements with an algebraic closure $\Fbar$, and let $X$ be a smooth projective variety over $\bF$; choose a prime $\ell$ different from the characteristic of the field $\mathbb{F}$. Let $\Fr\colon \Xbar \to \Xbar$ be the geometric Frobenius morphism and let $\Fr^*\colon {\rm H}_\et^\bullet\left(\Xbar,\bQ_\ell\right) \to {\rm H}_\et^\bullet\left(\Xbar,\bQ_\ell\right)$ be induced map on $\ell$-adic cohomology. Recall that the Lefschetz trace formula is the identity
\begin{equation} \label{trfr}
\#X(\bF) = \sum_i (-1)^i \Tr\left(\Fr^*\colon {\rm H}_\et^i\left(\Xbar,\bQ_\ell\right) \to {\rm H}_\et^i\left(\Xbar,\bQ_\ell\right) \right).
\end{equation}
By the Weil Conjectures, for every integer $i$, the linear map $\Fr^*\big|_{{\rm H}_\et^i\left(\Xbar,\bQ_\ell\right)}$ has integral characteristic polynomial and its eigenvalues have absolute value $q^{i/2}$.

We specialize now to the case in which $X$ is a del Pezzo surface of degree $2$.  The $\ell$-adic cohomology group ${\rm H}_\et^2\left(\Xbar,\bQ_\ell\right)$ is isomorphic to $\Pic(\Xbar) \otimes \bQ_\ell$ as a Galois module, and in particular it has rank $7$.  The group of lattice automorphisms of $\Pic(\Xbar)$ fixing the canonical class is $W(E_7)$, the Weyl group of $E_7$.  Since the action of the Frobenius endomorphism $\Fr^*$ preserves the intersection pairing on $\Pic(\Xbar)$ and fixes the canonical divisor class, it follows that the Frobenius endomorphism $\Fr^*$ acts on ${\rm H}_\et^2\left(\Xbar,\bQ_\ell\right)$ via an element of $W(E_7)$.  As a $W(E_7)$-module, the vector space $\Pic(\Xbar) \otimes_\bZ \bC$ is isomorphic to the direct sum of the trivial representation and the natural representation of $W(E_7)$ on the Cartan subalgebra of the Lie algebra $E_7$; we denote this representation by $\rho \colon W(E_7) \to \GL( \Pic(\Xbar) \otimes_\bZ \bC)$.

Since the cohomology groups ${\rm H}_\et^0\left(\Xbar,\bQ_\ell\right)$ and ${\rm H}_\et^4\left(\Xbar,\bQ_\ell\right)$ have dimension one, we deduce that the trace of the Frobenius endomorphism on these two groups gives a contribution of $q^2+1$ to the Lefschetz trace formula~\eqref{trfr}.  Moreover, since the cohomology group ${\rm H}_\et^2\left(\Xbar,\bQ_\ell\right)$ has dimension $7$ and contains a one-dimensional subspace with trivial action of the Frobenius endomorphism, we obtain that the contribution of this group to the Lefschetz trace formula is at least $-6q$.

\begin{lemma} \label{traces}
Let $\rho \colon W(E_7) \to \GL( \Pic(\Xbar) \otimes_\bZ \bC)$ be the natural representation of the Weyl group of $E_7$.  We have an equality of sets
\[
\{ \Tr (\rho(g)) \mid g \in W(E_7) \} = \{-6,-4,-3,-2,-1,0,1,2,3,4,5,6,8\};
\]
in particular, $-5$ does not occur as a trace in this representation.  We also have the equality
\[
\{ \Tr (\rho(g)) \mid g \in W(E_7) {\textup{ and }} ord(g) \equiv 0 \pmod{4} \} = \{-2,-1 ,0,1,2, 3,4\};
\]
in particular, the trace of any element of order divisible by four is at least $-2$.
\end{lemma}

\begin{proof}
See tables in the Atlas~\cite{atlas}.
\end{proof}

To prove the surface $X$ is unirational, we need to estimate the number of points on the ramification curve $R$, or equivalently on the branch curve $B$.  If the characteristic of $\mathbb{F}$ is two, then this curve is a plane conic, and therefore it has at most $2q+1$ points.  Otherwise, the curve $B$ is a smooth plane quartic and the following result of St\"ohr and Voloch applies.

\begin{lemma}[St\"ohr-Voloch~\cite{SV}] \label{stvo}
Let $C$ be a smooth plane quartic over a finite field $\mathbb{F}$ with $q$ elements of characteristic different from two.  Then the number of points of $C$ over $\mathbb{F}$ is at most $2q+6$, unless the field $\mathbb{F}$ has $9$ elements and the curve $C$ is isomorphic to the Fermat quartic with equation $x^4+y^4+z^4=0$.
\end{lemma}

\begin{proof}
See the discussion on~\cite{SV}*{p.~16}.
\end{proof}

We are ready to embark on a proof of Theorem~\ref{thm:main}.  We begin with a proposition that allows us to deal with many del Pezzo surfaces of degree two endowed with a conic bundle structure over their ground field.

\begin{proposition} \label{conifi}
Let $X$ be a del Pezzo surface of degree two over a finite field $\bF$. Suppose that $X$ admits a conic bundle structure, defined over $\bF$, and that the field $\bF$ has at least four, but not five, elements. Then the conic bundle has a smooth fiber, defined over $\bF$, and in particular, $X$ is unirational.
\end{proposition}

\begin{proof}
Let $\pi\colon X \to \bP^1$ be a conic bundle structure defined over $\bF$. The morphism $\pi$ has six distinct singular geometric fibers. If the field $\bF$ has at least $7$ elements, then there are smooth fibers of $\pi$ over $\bP^1(\bF)$. It remains to consider the case when $\bF$ has four elements. In this case there are five fibers over $\bP^1(\bF)$ and they cannot all be singular, since otherwise the remaining geometric singular fiber would be defined over $\bF$ and hence there would be six fibers of $\pi$ defined over $\bF$.  The final statement is a consequence of Theorem~\ref{unirat}.
\end{proof}

\begin{theorem} \label{7fuori}
Suppose that $\mathbb{F}$ is a finite field with at least $7$ elements and that $X$ is a del Pezzo surface of degree two containing a point outside the ramification divisor.  The surface $X$ is unirational.
\end{theorem}

\begin{proof}
Let $p$ be a point outside the ramification divisor.  By Theorem~\ref{ma2}, we reduce to the case in which $p$ is a generalized Eckardt point.  If one of the exceptional curves through $p$ is defined over the ground field $\mathbb{F}$, then the surface is not minimal and we conclude by~\cite{kollar}; if a pair of exceptional curves through $p$ is defined over the ground field $\mathbb{F}$, then the surface $X$ admits a conic bundle over $\mathbb{F}$ and we conclude by Proposition~\ref{conifi}.  Thus we reduce further to the case in which the Galois group acts transitively on the four exceptional curves through $p$.  In particular, the order of the Frobenius element $\Fr^*$ acting on $\Pic(\Xbar) \otimes _\bZ \bC$ is divisible by four, and by Lemma~\ref{traces} we deduce that the trace of $\Fr$ is at least $-2$.  

Combining the estimate on the trace of $\Fr$ with the inequality of Lemma~\ref{stvo}, we find that there are at least $16$ points on $X$ outside the ramification divisor, and we conclude applying Lemma~\ref{otto}.
\end{proof}

\begin{lemma} \label{mas8}
Let $X$ be a del Pezzo surface of degree two over $\bF_7$; if the ramification divisor contains at most $8$ points, then the surface $X$ contains a point not on the ramification divisor.
\end{lemma}

\begin{proof}
Let $\kappa \colon X \to \mathbb{P}^2$ be the morphism induced by the anticanonical divisor.  Let $C$ be the inverse image of a line $L$ in $\mathbb{P}^2$.  The curve $C$ is a curve of arithmetic genus one on $X$, and the possibilities for $C$ are as follows: 
\begin{enumerate}
\item \label{cali}
a geometrically integral smooth curve of genus one;
\item \label{casin}
a geometrically integral rational curve with exactly one node or one cusp;
\item 
a geometrically reducible curve whose geometric irreducible components are two smooth rational curves.
\end{enumerate}
If the curve $C$ is geometrically integral, then in case~\eqref{cali} it has at least $\lceil (\sqrt{7}-1)^2 \rceil = 3$ points by the Hasse-Weil bound, and in case~\eqref{casin} it has at least $7$ points.  At most $28$ lines in $\mathbb{P}^2$ have geometrically reducible inverse image, and at most $\binom{8}{2}=28$ lines in $\mathbb{P}^2$ contain at least two of the images of the points of the branch divisor of $\kappa$.  Since there are $57$ lines in $\mathbb{P}^2$, it follows that there is a line $L$ in $\mathbb{P}^2$ with the property that the curve $C=\kappa^{-1}(L)$ is geometrically irreducible and  contains at most one $\bF_7$-point on the ramification divisor.  By the previous analysis, we deduce that the curve $C$ contains at least two more points of $X$, that are therefore not contained in the ramification divisor.
\end{proof}

In the following theorem we show unirationality of del Pezzo surfaces of degree two over finite fields with at least $7$ elements.

\begin{theorem} \label{amano}
Let $\mathbb{F}$ be a finite field with at least $7$ elements and let $X$ be a del Pezzo surface of degree two over $\mathbb{F}$, then $X$ is unirational with at most one exception: the field $\mathbb{F}$ is isomorphic to the field $\mathbb{F}_9$ with $9$ elements and $X$ is isomorphic to the del Pezzo surface with equation 
\[
\nu w^2 = x^4 + y^4 + z^4 ,
\]
where $\nu $ is a non-square in $\mathbb{F}_9$.
\end{theorem}

\begin{proof}
Let $q$ be the number of elements of the field $\mathbb{F}$; we argue separately the three cases $q \geq 9$, $q=8$ and $q=7$.

\subsubsection*{Finite fields with at least 9 elements}
Let $\mathbb{F}$ be a finite field with $q \geq 9$ elements.  Using the Weil conjectures and Lemma~\ref{stvo} we estimate the number of points on $X$: the surface $X$ has at least $q^2-6q+1$ points, while the ramification curve has at most $2q+6$ points, unless the field $\mathbb{F}$ has 9 elements and the branch curve is isomorphic to the Fermat curve.  Apart from the exception, since the inequality $q \geq 9$ holds, there are points not on the ramification curve, and therefore we can apply Theorem~\ref{7fuori} to $X$ and deduce that $X$ is unirational.  In the exceptional case, up to isomorphism there are two possibilities for the double cover and only the one in the statement of the result does not have points outside the ramification curve.

\subsubsection*{The field with 8 elements}

Let $\mathbb{F}_8$ be a finite field with $8$ elements, and let $X$ be a del Pezzo surface of degree two over $\mathbb{F}_8$.  Since the characteristic of the ground field is two, the branch curve $B$ of $X$ is a plane conic.  If the curve $B$ has a geometrically irreducible component defined over the ground field, then we can apply Theorem~\ref{unirat} to obtain that the surface $X$ is unirational.  The only remaining possibility is that the curve $B$ is the union of two lines defined over a quadratic extension of $\mathbb{F}_8$.  In this case, the branch curve $B$ has a unique point defined over $\mathbb{F}_8$, and hence the surface $X$ has at least $8^2-6 \cdot 8 +1 - 1 = 16$ points not on the ramification curve and we can again conclude by applying Lemma~\ref{otto}.

\subsubsection*{The field with 7 elements}
Let $\bF_7$ be a finite field with $7$ elements, and let $X$ be a del Pezzo surface of degree two over $\bF_7$.  We show that the surface $X$ contains a point not on the ramification divisor; this is sufficient, by Theorem~\ref{7fuori}.

Let $\alpha$ be the trace of Frobenius acting on ${\rm H}^2_{\textup{\'et}}(\overline{X},\bQ_\ell)$.  By Lemma~\ref{stvo}, the number of points on the ramification divisor is at most $20$, and therefore the surface $X$ has a point outside of the ramification divisor, provided the inequality $7^2+7\alpha+1>20$ holds, or equivalently if the inequality $\alpha >-4-2/7$ holds.  Using Lemma~\ref{traces}, we deduce that it suffices to analyze the case in which $\alpha = -6$, and hence the number of points of $X$ is $7^2-6 \cdot 7 + 1 = 8$, by the Weil conjectures.  The possibility that the surface $X$ has exactly $8$ points all contained in the ramification divisor is excluded by Lemma~\ref{mas8}.
\end{proof}

\begin{remark}
Recall that a del Pezzo surface $X$ of degree two has at most $126$ generalized Eckardt points (\S\ref{ss:Eckardt}).  Therefore, imposing only that the number of points on $X$ outside the ramification curve exceeds $126$ in the Lefschetz trace formula, we would have obtained the bound $q \geq 16$ for the unirationality of $X$.  Using Lemma~\ref{otto} we reduce the bound from $16$ to $11$.  As we have seen in Theorem~\ref{amano}, there is a surface over the field with $9$ elements all of whose points lie on the ramification curve.
\end{remark}

This is as far as we were able to argue without using a computer.  To complete the proof of Theorem~\ref{thm:main}, we perform direct calculations on a computer, taking advantage of our work above.

\begin{proof}[Proof of Theorem~\ref{thm:main}]
If the field has at least $7$ elements, then the result follows by Theorem~\ref{amano}.  In the four cases in which the field $\mathbb{F}$ has $2,3,4,5$ elements, we use the computer program Magma~\cite{magma} to enumerate and analyze the possible surfaces.

Let $X$ be a del Pezzo surface over a finite field. The surface $X$ always has a point: this follows from an elementary adaptation of the Chevalley-Warning theorem; alternatively, it also follows 
from~\cite{Weil} or~\cite{Esnault}.  If the surface $X$ has a point that is not a generalized Eckardt point and is not contained in the ramification divisor, then the surface $X$ is unirational by Corollary~\ref{fima}.  If the surface $X$ has a point $p$ that is contained in the ramification divisor whose image under $\kappa$ is not contained in a bitangent line to the branch curve, then the surface is again unirational since the spine of $X$ at $p$ is a geometrically integral rational curve defined over the ground field and we can apply Theorem~\ref{unirat}.  Thus, to prove the result, we only need to consider surfaces $X$ whose points are either generalized Eckardt points or lie on the ramification divisor and the corresponding spine is a union of exceptional curves.  Furthermore, since del Pezzo surfaces of degree higher than two with a rational point are unirational, we may also exclude non-minimal del Pezzo surfaces in our considerations.

We now describe how we use a computer to complete the analysis. 
The functions we use are contained in the set of scripts ``{\tt UnirationalityCheck.txt}'' included at the end of the source file in the arXiv posting.

\subsubsection*{Fields $\mathbb{F}_2$ and $\mathbb{F}_4$}

Suppose that the del Pezzo surface $X$ is defined over a finite field $\mathbb{F}$ of characteristic two.  The ramification curve of $X$ is isomorphic to a plane conic, and if it contains a geometrically irreducible component, then we can apply Theorem~\ref{unirat} to conclude.  We reduce to the case in which the branch curve is the union of two conjugate lines defined over the degree two extension of the ground field.  Performing if necessary a linear change of coordinates of $\mathbb{P}^2$, we assume that the surface $X$ has an equation of the form $w^2+(x^2+\alpha xy + y^2) w + g(x,y,z) = 0$, where the polynomial $x^2+\alpha xy + y^2$ is irreducible over $\mathbb{F}$, and $g$ is a quartic form in $x,y,z$.

The function {\tt{Char2OnePt(q)}} first lists the surfaces over $\mathbb{F}_q$ with the following restrictions on the polynomial $g$:
\begin{itemize}
\item
there are no monomial of degree $3$ in $x$, since these terms can be eliminated using a substitution of the form $w \mapsto w+q(x,y,z)$ with $q$ a quadratic form; 
\item
if $c$ is the coefficient of $x^4$ or of $y^4$, then the polynomial $w^2+w+c$ does not admit a root in the ground field, since these roots correspond to points on the surface lying above the point $[1,0,0]$ or $[0,1,0]$.
\end{itemize}
The program then excludes those quartics corresponding to surfaces having a point lying above a point different from $[0,0,1]$ and further eliminates the singular surfaces.  No surface over $\mathbb{F}_2$ or $\mathbb{F}_4$ passes these tests.

We deduce that there are no del Pezzo surface of degree two over the finite fields $\mathbb{F}_2$ and $\mathbb{F}_4$ having only one rational point.  Since the branch curve of the surfaces we consider has only one rational point, it follows that the surface must have a point outside the ramification divisor.  By Theorem~\ref{ma2}, we reduce to the case in which the surface $X$ has a generalized Eckardt point $P$.  As explained in the Appendix, there is a change of coordinates of $\mathbb{P}^2$ fixing the equation of the branch curve and taking the point $\kappa(P) \in \mathbb{P}^2$ to the point $[1,0,0]$.  The function {\tt{CheckEck2(q)}} first lists all the quartics over $\mathbb{F}_q$ corresponding to surfaces with the point $[1,0,0]$ as a generalized Eckardt (see the Appendix); then it excludes the surfaces having at least $16$ points outside the ramification divisor (Lemma~\ref{otto}).  After removing the quartics defining singular surfaces, the program also removes those surfaces having a point outside the ramification divisor that is not a generalized Eckardt point. Finally, we exclude the surfaces for which the spine at the unique rational point of the ramification curve is geometrically irreducible (Theorem~\ref{unirat}). No surface over $\mathbb{F}_2$ or $\mathbb{F}_4$ passes these tests. We conclude that every del Pezzo surface of degree 2 over $\mathbb{F}_2$ or $\mathbb{F}_4$ is unirational.

\subsubsection*{Fields $\mathbb{F}_3$ and $\mathbb{F}_5$}
We use three separate functions to deal with surfaces defined over the fields $\mathbb{F}_3$ and $\mathbb{F}_5$: ``{\tt{CheckEck}}'', ``{\tt{CompileList}}'', ``{\tt{CheckOnlyBranchPoints}}''.

\subsubsection*{{\tt{CheckEck}}}
The function {\tt{CheckEck(q)}} first lists the del Pezzo surfaces of degree two over $\mathbb{F}_q$ of the form 
\[
w^2 = x^4 + Q_2(y,z) x^2 + Q_4(y,z) 
\]
having the point $[1,0,0,1]$ as a generalized Eckardt point.  Second, it excludes the surfaces $X$ whose branch curve contains a point where the tangent line is not a bitangent line: such a line lifts on $X$ to a geometrically integral rational curve, and unirationality follows from Theorem~\ref{unirat}.  Third, it excludes the surfaces admitting a non generalized Eckardt point outside of the ramification divisor: these surfaces are also unirational by Corollary~\ref{fima}.  Among the surfaces that are left, we exclude the ones that are not minimal.

Outcome:
\begin{itemize}
\item
over $\mathbb{F}_3$, no surface passes these tests;
\item
over $\mathbb{F}_5$, up to isomorphism, the only surface passing these tests is the surface 
\[
X/\bF_5 \colon \quad w^2 = x^4 + y^4 + z^4 , \hphantom{\quad \quad}
\]
which is unirational by Lemma~\ref{altra}.
\end{itemize}

\subsubsection*{{\tt{CheckOnlyBranchPoints}}}
The function {\tt CheckOnlyBranchPoints(q)} first lists the del Pezzo surfaces of degree two over $\mathbb{F}_q$ all of whose points project to points of bitangency on the branch curve.  We use the two normal forms determined by Shioda in~\cite{S} for the plane quartics with a point lying on a bitangent line: 
\begin{equation} \label{forno}
\begin{array}{r}
x^2 y^2 + z (x^3 + v_1 x^2 z + v_2 x y z + v_3 x z^2 + v_4 y^2 z + v_5 y z^2 + v_6 z^3 + v_7 y^3) = 0 , \\[5pt]
x^4 + y (z^3 + v_1 y^2 z + v_2 x y z + v_3 x^2 z + v_4 y^3 + v_5 x y^2 + v_6 x^2 y) = 0 .
\end{array}
\end{equation}
For the curves of the first kind, the point $[0,1,0]$ is a simple bitangency point and the corresponding bitangent line is the line with equation $z=0$.  For the curves of the second kind, the point $[0,0,1]$ is a flex bitangency point and the corresponding bitangent line is the line with equation $y=0$.  Each of the plane quartics with defining polynomial $F$ listed in~\eqref{forno} corresponds to two possible del Pezzo surfaces with equations $w^2 = F$ and $\nu w^2 = F$, where $\nu $ is a non-square in the ground field.  The function then excludes the surfaces admitting points outside the ramification curve, or admitting a point on the ramification curve that does not project to a point of bitangency on the branch curve.  Finally, the function excludes surfaces that are not minimal, keeping track of the remaining surfaces up to isomorphism.

Outcome:
\begin{itemize}
\item
over $\mathbb{F}_5$, there is no surface in the output of {\tt CheckOnlyBranchPoints(5)};
\item
over $\mathbb{F}_3$, up to isomorphism, the only two surfaces passing the tests contained in {\tt{CheckOnlyBranchPoints(3)}} are the surfaces 
\[
\begin{array}{lrcl}
X_1/\bF_3 \colon \quad & -w^2 & = & (x^2 + y^2)^2 + y^3z - yz^3, \\[5pt]
X_2/\bF_3 \colon \quad & -w^2 & = & x^4 + y^3z - yz^3.
\end{array}
\]
These are the two exceptions mentioned in the statement.\qed
\end{itemize}
\hideqed
\end{proof}

\begin{remark}
We computed the Galois cohomology group ${\rm H}^1\left(\Gal(\overline{\mathbb{F}} / \mathbb{F}), \Pic (\overline{X})\right)$ for the three exceptional surfaces of Theorem~\ref{thm:main}; the exponent of this group gives a lower bound for the degree of a dominant rational map $\mathbb{P}^2 \dashrightarrow X$ \cite[Theorem~29.2]{M}.  We obtained 
\begin{align*}
{\rm H}^1\left( \Gal(\overline{\mathbb{F}}_3 / \mathbb{F}_3) , \Pic (\overline{X}_1)\right) &\cong (\mathbb{Z}/2\mathbb{Z})^4, \\
{\rm H}^1\left( \Gal(\overline{\mathbb{F}}_3 / \mathbb{F}_3) , \Pic (\overline{X}_2)\right) &\cong (\mathbb{Z}/4\mathbb{Z})^2, \\
{\rm H}^1\left( \Gal(\overline{\mathbb{F}}_9 / \mathbb{F}_9) , \Pic (\overline{X}_3)\right) &\cong (\mathbb{Z}/2\mathbb{Z})^6.
\end{align*}
The Galois actions giving rise to these groups are quite special.  For example, for each of the groups $G \in \left\{(\mathbb{Z}/4\mathbb{Z})^2, (\mathbb{Z}/2\mathbb{Z})^6\right\}$, there is a unique conjugacy class of subgroups of $W(E_7)$ such that the corresponding Galois cohomology group is isomorphic to $G$.
\end{remark}

\begin{proof}[Proof of Corollary \ref{quadratica}]
Suppose first that the field $k$ is infinite. Then there is a point $P \in \mathbb{P}^2(k)$ that is contained neither on the branch locus of $\kappa$ nor on four bitangent lines. Passing at most to a quadratic extension $k'$ of $k$, we obtain that the point $P$ lifts to the surface $X \times_{k} k'$, and hence $X \times_{k} k'$ is unirational by Corollary~\ref{fima}.

Suppose now that the field $k$ is finite.  We begin by showing that there is a point $P \in \mathbb{P}^2(k)$ that is not contained in the branch curve $B$ of $\kappa$.  This is clear if the characteristic of the field $k$ is two, since in this case the curve $B$ is a conic.  If the characteristic of the field $k$ is odd, then we reduce to the case in which the curve $B$ contains a point $Q$.  The intersection of $B$ with tangent line to $B$ at $Q$ is a finite scheme of length four with $Q$ as a non-reduced point; in particular this intersection consists of at most $3$ reduced points over $k$.  Since the number of points on the tangent line to the curve $B$ at $Q$ contains at least four points, we deduce that there is indeed a point $P$ on $\mathbb{P}^2(k)$ not contained in the branch curve.  Over a quadratic extension $k'/k$ the point $P$ lifts to a point on $X \times_k k'$ outside the ramification divisor and we conclude applying Theorem~\ref{7fuori}, unless $k$ is the field $\mathbb{F}_2$.

Finally suppose that $k$ is $\mathbb{F}_2$.  The ramification curve acquires a geometrically integral component over at most a quadratic extension of the field $k$; in this case we conclude applying Theorem~\ref{unirat}.
\end{proof}

\section*{Appendix: generalized Eckardt points over fields of characteristic two}

Let $k$ be a field of characteristic two. Let $X$ be a del Pezzo surface of degree $2$ over $k$.  The surface $X$ is isomorphic to a hypersurface of degree $4$ in the weighted projective space $\bP(1,1,1,2)$ with equation 
\[
X \colon \quad w^2 + wf(x,y,z) + g(x,y,z) = 0,
\]
where $f(x,y,z)$ and $g(x,y,z)$ are homogeneous polynomials in $k[x,y,z]$ of respective degrees two and four.  The anticanonical morphism $\kappa\colon X \to \bP^2$ is obtained by forgetting the coordinate $w$ and an equation of the branch locus of the anticanonical morphism is $f(x,y,z)^2 = 0$.

Exceptional curves on $X$ are divisors of the form
\[
\left\{
\begin{array}{rcl}
\alpha(x,y,z) &=& w\\
\ell(x,y,z) &=& 0
\end{array}
\right\},
\]
where $\alpha(x,y,z)$ and $\ell(x,y,z)$ are forms in $k[x,y,z]$ of respective degrees two and one. After a linear change of variables, we assume that the coordinates of the generalized Eckardt point $P$ are $[1,0,0,w_0]$, so its projection to $\bP^2$ is $[1,0,0]$. To say $P$ is a generalized Eckardt point means that there are four members of the pencil of lines $\lambda y + \mu z = 0$ through $\kappa(P)$ that give rise to exceptional curves in $X$. On the open chart where $\mu = 1$, we let 
\[
\begin{array}{ccccl}
\alpha^{(\lambda)}(x,y) & = & \alpha(x,y,\lambda y) &=& \alpha_0 x^2 + \alpha_1 xy + \alpha_2 y^2, \\[5pt]
f^{(\lambda)}(x,y) & = & f(x,y,\lambda y) & = & f_0 x^2 + f_1 xy + f_2 y^2, \\[5pt]
g^{(\lambda)}(x,y) & = & g(x,y,\lambda y) & = & g_0 x^4 + g_1 x^3y + g_2 x^2y^2 + g_3 xy^3 + g_4 y^4,
\end{array}
\]
where the coefficients of the monomials in $x,y$ are polynomials in $\lambda$.  The line $z = \lambda y$ is a bitangent line if the polynomial
\[
\alpha^{(\lambda)}(x,y)^2 + \alpha^{(\lambda)}(x,y)f^{(\lambda)}(x,y) + g^{(\lambda)}(x,y)
\]
vanishes identically, as a polynomial in $x$ and $y$.  Thus we obtain the system of equations 
\begin{equation} \label{cona0}
\left\{
\begin{array}{rcl}
\alpha_0^2 + \alpha_0f_0 + g_0 &=& 0 \\[5pt]
\alpha_0f_1 + \alpha_1f_0 + g_1 &=& 0 \\[5pt]
\alpha_1^2 + \alpha_0f_2 + \alpha_1f_1 + \alpha_2f_0 + g_2 &=& 0\\[5pt]
\alpha_1f_2 + \alpha_2f_1 + g_3 &=& 0\\[5pt]
\alpha_2^2 + \alpha_2f_2 + g_4 &=& 0 .
\end{array}
\right.
\end{equation}
We use the relations $\alpha_0f_1 = \alpha_1f_0 + g_1$ and $\alpha_2f_1 = \alpha_1f_2 + g_3$ to eliminate $\alpha_0,\alpha_2$ from the first, third and fifth equations in~\eqref{cona0}, obtaining the system 
\begin{equation} \label{quadra}
\left\{
\begin{array}{rcl}
\alpha_1^2f_0^2 + \alpha_1f_0^2f_1 + (g_1^2 + f_0f_1g_1 + f_1^2g_0) & = & 0\\[5pt]
\alpha_1^2f_1 + \alpha_1f_1^2 + (f_0g_3 + f_1g_2 + f_2g_1) & = & 0 \\[5pt]
\alpha_1^2f_2^2 + \alpha_1f_1f_2^2 + (f_1^2g_4 + f_1f_2g_3 + g_3^2) & = & 0.
\end{array}
\right.
\end{equation}
Define 
\[
\begin{array}{rcl}
e_0 & := & g_1^2+f_0f_1g_1+f_1^2g_0 \\[5pt]
e_1 & := & f_0g_3 + f_1g_2 + f_2g_1 \\[5pt]
e_2 & := & f_1^2g_4 + f_1f_2g_3 + g_3^2 ;
\end{array}
\]
viewing the equations in~\eqref{quadra} as quadratic in $\alpha_1$, we find that the relations
\begin{equation}
\label{eq:rels}
e_0f_2^2 - e_2f_0^2 \; = \; e_0f_1 - e_1f_0^2 \; = \; e_1f_2^2 - e_2f_1 \; = \; 0
\end{equation}
hold.  Write
\[
\begin{split}
g(x,y,z) &= a_{1}x^4 + a_{2}x^3y + a_{3}x^3z + a_{4}x^2y^2 + a_{5}x^2yz + a_{6}x^2z^2 + a_{7}xy^3 +  \\
&\quad a_{8}xy^2z + a_{9}xyz^2 + a_{10}xz^3 + a_{11}y^4 + a_{12}y^3z + a_{13}y^2z^2 + a_{14}yz^3 + a_{15}z^4
\end{split}
\]
and
\[
f(x,y,z) = a_{16}x^2 + a_{17}xy + a_{18}xz + a_{19}y^2 + a_{20}yz + a_{21}z^2 ;
\]
we express the relations~\eqref{eq:rels} in terms of $a_1,\ldots,a_{21}$ and $\lambda$.  The relation $e_0f_1 - e_1f_0^2 = 0$ has degree $3$ as a polynomial in $\lambda$, and if $P$ is a generalized Eckardt point, then every coefficient of this polynomial must be zero. We find the relations 
\begin{equation} \label{tuttea}
\left\{
\begin{array}{rcl}
a_{7}a_{16}^3 & = & a_{1}a_{17}^3 + a_{2}^2a_{17} + a_{2}a_{16}^2a_{19} + a_{2}a_{16}a_{17}^2 + a_{4}a_{16}^2a_{17} \\[5pt]
a_{8}a_{16}^3 & = & a_{1}a_{17}^2a_{18} + a_{2}^2a_{18} + a_{2}a_{16}^2a_{20} + a_{3}a_{16}^2a_{19} + a_{3}a_{16}a_{17}^2 + \\[3pt]
& & a_{4}a_{16}^2a_{18} + a_{5}a_{16}^2a_{17} \\[5pt]
a_{9}a_{16}^3 & = & a_{1}a_{17}a_{18}^2 + a_{2}a_{16}^2a_{21} + a_{2}a_{16}a_{18}^2 + a_{3}^2a_{17} + a_{3}a_{16}^2a_{20} + \\[3pt]
& & a_{5}a_{16}^2a_{18} + a_{6}a_{16}^2a_{17} \\[5pt]
a_{10}a_{16}^3 & = & a_{1}a_{18}^3 + a_{3}^2a_{18} + a_{3}a_{16}^2a_{21} + a_{3}a_{16}a_{18}^2 + a_{6}a_{16}^2a_{18} .
\end{array}
\right.
\end{equation}
Note that these relations are linear in the variables $a_7,\dots,a_{10}$. We normalize the equation of $X$ so that $a_{16} = 1$ always, and hence the coefficients of $xy^3$, $xy^2z$, $xyz^2$ and $xz^3$ are determined by the remaining coefficients.

We specialize the discussion above and use the automorphism group of $\mathbb{P}^2$ to reduce the possibilities for the branch curve.  Up to a linear change of coordinates, the polynomial $f$ is proportional to a polynomial in the list 
\begin{equation} \label{effe}
\begin{array}{rl}
    x^2 + xy + z^2 & \quad \textup{(smooth conic),} \\
\hphantom{irreducible}    x^2 + \alpha xy + y^2 & \quad \textup{(irreducible, geometrically reducible conic),} \\
    x^2 + xy & \quad \textup{(reducible conic),} \\
    x^2 & \quad \textup{(non-reduced conic),}
\end{array}
\end{equation}
where $\alpha$ is an element of $k$ such that $\alpha \notin \{t^2 + t : t \in k\}$.  In the case of a smooth conic, the automorphisms of $\bP^2$ that stabilize the conic act on the points of $\mathbb{P}^2$ outside the conic with two orbits: the strange point of the conic (the point contained in every tangent line of the conic) and everything else.  Since the point $[1,0,0]$ is not the strange point of $x^2 + xy + z^2 = 0$, we may perform the transformation of $f$ without perturbing our desired Eckardt point.  In the remaining cases, the group of automorphisms of $\mathbb{P}^2$ stabilizing the vanishing set of $f$ act transitively on the complement of $V(f)$.  Taking advantage of the possibility of rescaling $w$ and renormalizing the equation of $X$, we reduce to the case in which the polynomial $f$ is one of the polynomials in~\eqref{effe} and the coordinates of the point $\kappa(P)$ are $[1,0,0]$.

\begin{example}[Smooth conic]
If $f(x,y,z) = x^2 + xy + z^2$, then the equations~\eqref{tuttea} specialize~to 
\[
\left\{
\begin{array}{rcl}
a_{7} & = & a_{1} + a_{2}^2 + a_{2} + a_{4} \\
a_{8} & = & a_{3} + a_{5} \\
a_{9} & = & a_{2} + a_{3}^2 + a_{6} \\
a_{10} & = & a_{3} .
\end{array}
\right.
\]
\end{example}

\begin{example}[Irreducible, geometrically reducible conic]
If $f(x,y,z) = x^2 + \alpha xy + y^2$, then the equations~\eqref{tuttea} specialize to 
\[
\left\{
\begin{array}{rcl}
a_{7} & = & \alpha^3 a_{1} + \alpha a_{2}^2 + a_{2} + \alpha^2 a_{2} + \alpha a_{4} \\
a_{8} & = & a_{3} + \alpha^2 a_{3} + \alpha a_{5} \\
a_{9} & = & \alpha a_{3}^2 + \alpha a_6 \\
a_{10} & = & 0 .
\end{array}
\right.
\]
\end{example}

\begin{example}[Reducible conic]
If $f(x,y,z) = x^2 + xy$, then the equations~\eqref{tuttea} specialize~to 
\[
\left\{
\begin{array}{rcl}
    a_7 & = & a_{1} + a_{2}^2 + a_{2} + a_{4}\\
    a_8 & = & a_{3} + a_{5}\\
    a_9 & = & a_{3}^2 + a_{6}\\
    a_{10} & = & 0.
\end{array}
\right.
\]
\end{example}

\begin{example}[Double line]
If $f(x,y,z) = x^2$, then the equations~\eqref{tuttea} specialize to 
\[
\left\{
\begin{array}{rcl}
    a_7 & = & 0\\
    a_8 & = & 0\\
    a_9 & = & 0\\
    a_{10} & = & 0.
\end{array}
\right.
\]
\end{example}

\end{document}